\documentclass[12pt]{amsart}
\usepackage{mathrsfs}
\usepackage{txfonts}
\allowdisplaybreaks[4]%also pxfonts
\usepackage[colorlinks, citecolor=blue,pagebackref,hypertexnames=false]{hyperref}
\usepackage{pgf,tikz}

\theoremstyle{plain}
\newtheorem{theorem}{Theorem}[section]
\newtheorem{prop}[theorem]{Proposition}
\newtheorem{lemma}[theorem]{Lemma}

\theoremstyle{definition}
\newtheorem{definition}[theorem]{Definition}

\def\RR{\Bbb R}
\def \R   {\mathbb{R}^n}

\def \UHRN {X\times (0,+\infty)}

\def \HMAX {H^p_{L,max}(X)}

\allowdisplaybreaks

\title[Maximal function characterizations for  Hardy spaces]
{Maximal function characterizations for  Hardy spaces\\  associated to nonnegative self-adjoint
operators\\ on spaces of homogeneous type}

\author{Liang Song \ and \ Lixin Yan  }
\address{
   Liang Song,
    Department of Mathematics,
    Sun Yat-sen University,
    Guangzhou, 510275,
    P.R.~China}
\email{ songl@mail.sysu.edu.cn}

\address{
    Lixin Yan,
    Department of Mathematics,
    Sun Yat-sen University,
    Guangzhou, 510275,
    P.R.~China}
\email{
    mcsylx@mail.sysu.edu.cn}



\begin{document}

\begin{abstract}
Let $X$ be a metric measure space   with a  doubling measure  and
  $L$ be a nonnegative self-adjoint operator
  %of second order
  acting on $L^2(X)$. Assume that $L$ generates
an analytic   semigroup $e^{-tL}$ whose kernels $p_t(x,y)$ satisfy
Gaussian upper bounds but  without any assumptions on the regularity
of space variables $x$ and $y$.
In this article we   continue  a study in  \cite{SY} to
  give an atomic decomposition for the Hardy spaces
 $ H^p_{L,max}(X)$  in terms of the nontangential
 maximal function  associated with the heat semigroup
 of $L$,  and hence we   establish
 characterizations of    Hardy     spaces associated to an operator $L$,
 via  an atomic   decomposition or the nontangential maximal function.
   We also obtain  an equivalence of
    $ H^p_{L, max}(X)$ in terms of  the radial maximal function.
\end{abstract}

\maketitle

 %\tableofcontents

%-----------------------------------------------------------
\section{Introduction}\label{sec:intro}
\setcounter{equation}{0}

Our goal in this paper is to   continue  a study in  \cite{SY} to establish the equivalence
of the maximal and atomic Hardy  spaces on spaces of homogeneous type, associated to nonnegative self-adjoint operators whose
heat kernel has Gaussian upper bounds. For the theory of Hardy spaces associated to operators,
it has attracted a lot of attention in the last   decades,
 and  has been a  very active research topic  in harmonic analysis -- see for example,
    \cite{ADM, AMR,  AR, CKS,  DKKP, DL, DY, DZ, HLMMY, HM, HMMc, JY,  SY, Y, YY}.

Let  $(X,d,\mu)$ be a metric measure space
endowed with a distance $d$
and a nonnegative Borel doubling
measure $\mu$ on $X$ (\cite{CW}).
Recall that a measure $\mu$ is doubling provided that there exists a
constant $C>0$ such that for all $x\in X$ and for all $r>0$,
\begin{eqnarray}
V(x,2r)\leq C V(x,r)
\label{e1.1}
\end{eqnarray}
where $V(x,r)=\mu(B(x,r))$, the volume of the open ball  $B =
B(x,r) := \{y \in X: d(y, x)<r\}$.
 Note that the doubling property  implies the following
strong homogeneity property,
\begin{equation}
V(x, \lambda r)\leq C\lambda^n V(x,r)
\label{e1.2}
\end{equation}

\noindent for some $C, n>0$ uniformly for all $\lambda\geq 1, r>0$ and $x\in X$.
In Euclidean space with Lebesgue measure, the parameter $n$ corresponds to
the dimension of the space, but in our more abstract setting, the optimal $n$
 need not even be an integer. There also
exists  $C>0$  so that

\begin{equation}
V(y,r)\leq C\Big( 1+\frac{d(x,y)}{r}\Big)^n V(x,r)
\label{e1.3}
\end{equation}

\noindent
uniformly for all $x,y\in X$ and $r>0$. Indeed, property (\ref{e1.3})  is a direct
consequence of the triangle inequality for the metric
$d$ and the strong homogeneity property (\ref{e1.2}).

The following will be assumed throughout the article unless otherwise specified:

\medskip

 \noindent
 {\bf   (H1)} \
  $L$ is a non-negative self-adjoint operator on $L^2({X})$;

\smallskip

 \noindent
 {\bf   (H2)} \ The kernel of $e^{-tL}$, denoted by $p_t(x,y)$,
 is a measurable function on $X\times X$ and  satisfies
a Gaussian upper bound, that is
$$ \leqno{\rm (GE)}\hspace{4cm}
 |p_t(x,y)|\leq C \frac{1}{V(x,\sqrt{t})} \exp\left(-{  {{d(x,y)}^2}\over  ct}\right)
$$
for all $t>0$,  and $x,y\in X,$   where $C$ and $c$   are positive
constants.

We now recall the notion of a
 $(p,q,M)$-atom associated to
an operator $L$ (\cite{AMR, DL, HLMMY}).

\begin{definition}\label{def1.2}
Given   $0<p\leq 1\leq q\leq \infty$, $p<q$  and   $M\in {\mathbb N}$,
a function $a\in L^2(X)$ is called a $(p,q,M)$-atom associated to
the operator $L$ if there exist a function $b\in {\mathcal D}(L^M)$
and a ball $B\subset X$ such that

\smallskip

{  (i)}\ $a=L^M b$;

\smallskip

{  (ii)}\ {\rm supp}\  $L^{k}b\subset B, \ k=0, 1, \dots, M$;

\smallskip

{  (iii)}\ $\|(r_B^2L)^{k}b\|_{L^q (X)}\leq r_B^{2M}
{V(B)}^{1/q-1/p},\ k=0,1,\dots,M$.
\end{definition}

\smallskip

The  atomic Hardy space $H^p_{L, {\rm at}, q, M}(X)
 $ is defined as follows.

\begin{definition} \label{def1.3}
 We will say that $f= \sum\lambda_j
a_j$ is an atomic $(p, q, M)$-representation (of $f$) if $
\{\lambda_j\}_{j=0}^{\infty}\in {\ell}^p$, each $a_j$ is a
$(p, q, M)$-atom, and the sum converges in $L^2(X).$  Set
\begin{eqnarray*}
\mathbb  H^p_{L, {\rm at}, q, M}(X):=\Big\{f:  f \mbox{  has
an atomic $(p, q, M)$-representation} \Big\},
\end{eqnarray*}
with the norm $\big\|f\big\|_{\mathbb  H^p_{L, {\rm at}, q, M}(X)}$ given by
\begin{eqnarray*}
\inf\Big\{\Big(\sum_{j=0}^{\infty}|\lambda_j|^p\Big)^{1/p}:
f=\sum\limits_{j=0}^{\infty}\lambda_ja_j \ \mbox{ is an atomic
$(p,q,M)$-representation}\Big\}.
\end{eqnarray*}
The space $H^p_{L, {\rm at}, q, M}(X)$ is then defined as
the completion of $\mathbb  H^p_{L, {\rm at}, q, M}(X)$ with
respect to this norm.
 \end{definition}

%Obviously, $H^p_{L, {\rm at}, q_2, M}(X)\subseteq H^p_{L, {\rm at}, q_1, M}(X) $  when $1\leq q_1\leq q_2\leq \infty$.

Given a function $f\in L^2({X})$, consider the following  non-tangential maximal
function associated to  the heat semigroup generated by the
 operator $L$,
 \begin{eqnarray}\label{e1.4}
 f^*_L(x)=:\sup\limits_{d(x,y)<t}|e^{-t^2L}f(y)|.
\end{eqnarray}
 We may define the spaces  $ H^p_{L,max}({X}), 0<p\leq 1 $
as the completion of $\{L^2({X}): \|f^*_L\|_{L^p({X})}<\infty\}$
with respect to $L^p$-norm of the non-tangential maximal function; i.e.,
\begin{eqnarray}\label{e1.5}
\big\|f\big\|_{H^p_{L, max}({X})}:=\big\|f^*_L\big\|_{L^p({X})}
\end{eqnarray}
It can be verified (see \cite{HLMMY, DL, JY}) that for all $q>p$ with   $1\leq q\leq \infty$ and every
number $M>\frac{n}{2}(\frac{1}{p}-1)$, any $(p,q,M)$-atom $a$ is in $ H^p_{L,max}({X})$
  and so the following
continuous inclusion  holds:
\begin{eqnarray}\label{e1.6}
H^p_{L, {\rm at}, q, M}({X})\subseteq  H^p_{L, max}({X}).
\end{eqnarray}
A natural question is to show the following continuous inclusion
 $ H^p_{L, max}({X})\subseteq H^p_{L, {\rm at}, q, M}({X})$.
In the case of  $X={\mathbb R^n}$, it is known that
   the inclusion $ H^p_{L, max}({\mathbb R^n})\subseteq
H^p_{L, {\rm at}, q, M}({\mathbb R^n})$ holds  for certain    operators including
 Schr\"odinger operators   with nonnegative potentials  and second order divergence form elliptic operators
 via particular PDE technique  (see for example, \cite{DZ, FS, HLMMY, HM}).
Very recently,   the authors of this article have made  a reformulation and modification of
 of a technique due to A. Calder\'on \cite{C} to
obtain an atomic decomposition directly from
 $ H^p_{L,max}(\R)$.
 Precisely,
  %and this leads   eventually to characterizations of Hardy spaces associated to $L$,  via
  %atomic decomposition or the nontangential  maximal functions. That is,
under the assumptions   ${\bf   (H1)}$  and ${\bf   (H2)}$ of the operator $L$
 but  without any assumptions on the regularity of $p_t(x,y)$
of space variables $x$ and $y$,
we  have  that   $H^p_{L, max}({\mathbb R^n})\subseteq
H^p_{L, {\rm at}, q, M}({\mathbb R^n})$ for   $0<p\leq 1\leq q\leq \infty$ with $q>p$,
and    all integers $M> {n\over 2}({1\over p}-1)$,
and hence by \eqref{e1.6},
$$
 H^p_{L, {  max}}({\mathbb R^n}) \simeq     H^p_{L, {\rm at}, q, M}({\mathbb R^n}).
$$
We point out that in \cite{C},
%the proof of in  \cite[Theorem 1.4]{SY},
a decomposition  of the function
  $ F(x,t)=f\ast \varphi_t(x)$ associated with the distribution $f$
 was given, and  convolution operation of the function $F$ played an important role in the proof.
 In \cite[Theorem 1.4]{SY},  no analogue of
convolution operation of the function  $ t^2L e^{-t^2L}f(x)$,
the proof
%of \cite[Theorem 1.4]{SY}
 depends    critically on the geometry of ${\mathbb R^n}$ to  use
  oblique cylinders  of ${\mathbb R}^{n+1}_+$:
   for every  cube $Q$ of $\mathbb{R}^n$ and for $\bar{e}=(1,\cdots,1)\in \R$
   %  the authors in  \cite{SY} define a kind of  ``oblique cylinder",
\begin{align*}
\tilde{Q}:=\{({y}, t)\in \mathbb{R}^{n+1}_+: {y}+3t\bar{e}\in Q\},
\end{align*}
%where  $\bar{e}=(1,\cdots,1)\in \R$.
 in place of vertical cylinders in
  Calder\'on's construction in \cite{C}.
  \iffalse
  Our proof is  based on a modification of a technique due to A. Calder\'on \cite{C},
 where a decomposition  of the function $ F(x,t)=f\ast \varphi_t(x)$ associated with the distribution $f$
 was given in \cite{C}, and convolution operation of the function $F$ played an important role in the proof.
  In our setting,  there is, however,
  no analogue of
convolution operation of the function  $ t^2L e^{-t^2L}f(x)$, we have to modify Calder\'on's
construction and
 the geometry   in conducting the
analysis
   More precisely,  for every  cube $Q$ of $\mathbb{R}^n$,   the authors in  \cite{SY} define a kind of  ``oblique cylinder",
\begin{align*}
\tilde{Q}:=\{({y}, t)\in \mathbb{R}^{n+1}_+: {y}+3t\bar{e}\in Q\},
\end{align*}
where  $\bar{e}=(1,\cdots,1)\in \R$. In order to obtain an atomic decomposition  from
the maximal Hardy spaces, they  decompose
$\mathbb{R}^{n+1}_+$  into  ``oblique cylinders". \fi
  However, ``oblique cylinders"  do not exist on spaces of homogeneous type
  and hence
   it is not trivial to  generalize the method in \cite{SY} to the case of spaces of homogeneous type.
  So, we may ask the following question:

\medskip

\noindent
{\large \bf Question 1.} \ Is it possible to show an inclusion   $H^p_{L, max}(X)\subseteq
H^p_{L, {\rm at}, q, M}(X)$ on space of homogeneous type $X$?

 \medskip

Next we consider the Hardy space
 $ H^p_{L, max}(X)$  in terms of the radial maximal function.
 Given an operator    $L$ satisfying ${\bf (H1)}$-${\bf (H2)}$,
we may define the spaces $H^p_{L,\ rad}(X), 0<p\leq 1 $
as the completion of $\{f \in L^2(X): \|f^+_L\|_{L^p(X)}<
\infty\}$ with respect to $L^p$-norm of the radial  maximal function; i.e.,
\begin{eqnarray}\label{e1.7}
\big\|f\big\|_{H^p_{L,  rad}(X)}:=\big\|f^+_L\big\|_{L^p(X)}:=\left\|\sup\limits_{t>0}|e^{-t^2L}f|\right\|_{L^p(X)} .
\end{eqnarray}
Fix $0<p\leq 1$. For all $q>p$ with $1\leq q\leq \infty$ and  for all integers $M> {n\over 2}({1\over p}-1)$,
 %any $(p, q, M)$-atom $a$ is in $ H^p_{L, rad}(X)$ (see \cite{HLMMY}), and so
  the following
continuous inclusion  holds:
\begin{eqnarray}\label{e1.8}
H^p_{L, {\rm at}, q, M}({X})\subseteq  H^p_{L, max}({X})\subseteq  H^p_{L,  rad}(X)
\end{eqnarray}
  by  \eqref{e1.4}-\eqref{e1.7}.
In the case of  $X={\mathbb R^n}$, it is known that  the inclusion $ H^p_{L, rad}({\mathbb R^n})\subseteq
H^p_{L, max}({\mathbb R^n})$   holds  for certain    operators including
 Schr\"odinger operators   with nonnegative potentials and second order divergence form elliptic operators
 via particular PDE technique  (see example, \cite{DZ, FS, HLMMY}).
\iffalse
Fix $0<p\leq 1$. For all $q>p$ with $1\leq q\leq \infty$ and  for all integers $M> {n\over 2}({1\over p}-1)$,
 %any $(p, q, M)$-atom $a$ is in $ H^p_{L, rad}(X)$ (see \cite{HLMMY})  and so
 the following
continuous inclusion  holds:
\begin{eqnarray}\label{e1.8}
H^p_{L, {\rm at}, q, M}(X)\subseteq  H^p_{L,  rad}(X).
\end{eqnarray}
As above,
   the inclusion $ H^p_{L, rad}({\mathbb R^n})\subseteq
H^p_{L, {\rm at}, q, M}({\mathbb R^n})$ also holds  for certain    operators including
 Schr\"odinger operators   with nonnegative potentials and second order divergence form elliptic operators
 via particular PDE technique  (see example, \cite{DZ, FS, HLMMY}).
%However,
% this question is still open  assuming merely  that an operator $L$ satisfies  ${\bf (H1)}$-${\bf (H2)}$.
\fi
 We may ask the following question:

\medskip

\noindent
{\large \bf Question 2.} \  Is it possible to show an inclusion   $H^p_{L, rad}(X)\subseteq
H^p_{L, max}(X)$  assuming merely  that an operator $L$ satisfies  ${\bf (H1)}$-${\bf (H2)}$?

 \medskip

 The aim of this article  is  give an affirmative answer to Questions 1 and 2 above to
establish  the equivalent
  characterizations
 of Hardy     space
associated to an operator $L$ satisfies ${\bf   (H1)}$  and ${\bf   (H2)}$ on spaces of homogeneous type $X$,
 including an atomic   decomposition,
   the nontangential maximal functions and the radial maximal functions.
Throughout  the article, we always assume that $\mu(X)=\infty$ and $\mu(\{x\})=0$ for all $x\in X$.

\begin{theorem}\label{th1.1}
Let $(X,d,\mu)$ be as in \eqref{e1.1} and \eqref{e1.2}.
Suppose  that an operator $L$ satisfies ${\bf   (H1)}$  and ${\bf   (H2)}$.
Fix $0<p\leq 1$. For all $q>p$ with $1\leq q\leq \infty$ and  for all integers $M> {n\over 2}({1\over p}-1)$,
we  have that
 \begin{itemize}
\item[(i)] $H^p_{L, rad}({X})\subseteq H^p_{L, max}({X});$

\item[(ii)]  $ H^p_{L,max}({X})\subseteq
H^p_{L, {\rm at}, q, M}({X})$.
\end{itemize}
Hence by \eqref{e1.8},
$$
 H^p_{L, {\rm at}, q, M}({X})  \simeq  H^p_{L, { max}}({X}) \simeq  H^p_{L, {  rad}}({X}).
$$
\end{theorem}

	The layout of the article is as follows. In Section 2 we recall some basic properties of
heat kernels and finite propagation speed for the wave equation, and build the necessary kernel estimates
for functions of an operator, which is useful in
 the proof of   Theorem~\ref{th1.1}. In Section 3
 we show  (i) of Theorem~\ref{th1.1} to obtain an equivalence of Hardy spaces on spaces of homogeneous type,  in terms of
   the nontangential and  radial maximal functions.
   In Section 4 we will show  our main result, (ii) of Theorem~\ref{th1.1}.
A crucial idea in the proof is to
 make  a   modification of \cite{C, SY}
 to   decompose  $X\times (0, \infty)$  into   type of ``cones"
 in place of ``vertical cylinders" of $\mathbb{R}^{n+1}_+$  in \cite{C} or
  ``oblique cylinders" of $\mathbb{R}^{n+1}_+$ in \cite{SY} (see   \eqref{eoo} below).
  This leads us to obtain
 characterizations of    Hardy     spaces associated to an operator $L$ on spaces of homogeneous type,
 via  an atomic   decomposition or the nontangential maximal function.

Throughout, the letter ``$c$"  and ``$C$" will denote (possibly
different) constants  that are independent of the essential
variables.

\medskip

\bigskip

\section{Preliminaries}
\setcounter{equation}{0}

Recall that, if  $L$ is a nonnegative,
self-adjoint operator on $L^2({X})$, and    $E_L(\lambda)$ denotes a spectral
decomposition associated with $L$,   then for every bounded Borel function
$F:[0,\infty)\to{\Bbb C}$, one defines the operator
$F(L): L^2({X})\to L^2({X})$ by the formula
\begin{align}\label{e2.1}
F(L):=\int_0^{\infty}F(\lambda) \, {\rm d}E_L(\lambda).
\end{align}
In particular, the  operator $ \cos(t\sqrt{L})$  is then well-defined  on
$L^2({X})$. Moreover, it follows from Theorem 3.4 of
\cite{CS}
  that
 the integral kernel
 $K_{\cos(t\sqrt{L})}$ of $\cos(t\sqrt{L})$ satisfies
\begin{align}
 {\rm supp}K_{\cos(t\sqrt{L})}\subseteq
 \bigl\{(x,y)\in {X}\times {X}: d(x,y)\leq   t\bigr\}.
 \label{e2.2}
 \end{align}
  By the Fourier inversion formula, whenever $F$ is
an even bounded Borel function with the Fourier transform of $F$,
$\hat{F} \in L^1(\mathbb{R})$, we can  write $F(\sqrt{L})$ in terms
of $\cos(t\sqrt{L})$. Concretely, by recalling (\ref{e2.1}) we have
$$
F(\sqrt{L})=(2\pi)^{-1}\int_{-\infty}^{\infty}{\hat
F}(t)\cos(t\sqrt{L}) \, {\rm d}t,
$$
which, when combined with (\ref{e2.2}), gives
\begin{align}\label{e2.3}
K_{F(\sqrt{L})}(x,y)=(2\pi)^{-1}\int_{|t|\geq  d(x,y)}{\hat F}(t)
K_{\cos(t\sqrt{L})} (x,y) \, {\rm d}t.
\end{align}
This property leads us the following result (see \cite[Lemma 3.5]{HLMMY}):
 For every  even function $\varphi\in C_0^{\infty}(\mathbb R)$
with supp $\varphi\subset (-1, 1)$,
\begin{align}
 {\rm supp}K_{\Phi(t\sqrt{L})}\subseteq
 \bigl\{(x,y)\in {X}\times {X}: d(x,y)\leq   t\bigr\},
 \label{e2.22}
 \end{align}
 where $\Phi$ denotes the Fourier transform of $\varphi$.

\begin{lemma}\label{le2.1} Assume that an operator $L$ satisfies  ${\bf (H1)}$-${\bf (H2)}$.
\begin{itemize}

\item[(i)] Let $\varphi \in {\mathscr S}(\mathbb R) $ be an even function.
   Then for every $\beta>0$, there exists a positive constant $C=C({n, \beta, \varphi})$
such that the kernel $K_{\varphi(t\sqrt{L})}(x,y)$ of $ \varphi(t\sqrt{L})$
satisfies
\begin{eqnarray}\label{e2.4}
\big|K_{\varphi(t\sqrt{L}) }(x,y)\big|
\leq  C\,   \frac{1}{\max(V(x,t), V(y,t))}
 \left(1+ \frac{d(x,y)}{t}\right)^{-n-\beta}
\end{eqnarray}
for all $t>0$ and $x, y\in {X}$.

\item[(ii)]
Let $\psi_i \in {\mathscr S}(\mathbb R) $ be even functions, $\psi_i(0)=0, i=1,2$.
   Then for every $\beta>0$, there exists a positive constant $C=C({n, \beta, \psi_1, \psi_2})$
such that the kernel $K_{\psi_1(s\sqrt{L}) \psi_2(t\sqrt{L}) }(x,y)$ of $ \psi_1(s\sqrt{L}) \psi_2(t\sqrt{L})$
satisfies
\begin{eqnarray}\label{e2.5}\\
\hspace{0.5cm}\big|K_{\psi_1(s\sqrt{L}) \psi_2(t\sqrt{L}) }(x,y)\big|
\leq  C\, \min\Big(\frac{s}{t}, \frac{t}{s}\Big)\,   \frac{1}{\max\left(V(x,\max{(s,t)}), V(y,\max{(s,t)}\right)}
 \left(1+ \frac{d(x,y)}{\max{(s,t)}}\right)^{-n-\beta}\nonumber
\end{eqnarray}
for all $t>0$ and $x, y\in {X}$.
\end{itemize}
\end{lemma}

\begin{proof} The proof  of (i) and (ii) is similar to that of \cite[Lemma 2.3]{CDa}
and \cite[Lemma 2.3]{SY} on the Euclidean spaces ${\mathbb R}^n$, respectively. We omit the detail
here.
 \end{proof}

\medskip

Let  $F(y,t)$ be a $\mu$-measurable function of $X\times (0,+\infty)$.
For $\alpha>0,$ set $F^{\ast}_{\alpha}(x)=\sup\limits_{d(x,y)<\alpha t}|F(y,t)|$.
 With the notation above, we have the following result.
\begin{lemma}\label{le2.2}
For any $p>0$  and $0<\alpha_2\leq \alpha_1$,
$$
\left\|F^{\ast}_{\alpha_1}\right\|_{L^p({X})}\leq
C\left(1+{2\alpha_1\over\alpha_2}\right)^{n/p}\left\|F^{\ast}_{\alpha_2}\right\|_{L^p({X})},
$$
where $C=C(p,n)$ is independent of $\alpha_1,\alpha_2 $ and $F$.
\end{lemma}
\begin{proof}
   The proof of Lemma~\ref{le2.2} is standard (see for instance,
  \cite[Theorem 2.3]{CT} for the case of $X=\mathbb{R}^n$).
  %We give a brief argument of this proof for completeness and convenience for the reader.
We write
$$
\left\|F^{\ast}_{\alpha_1}\right\|^p_{L^p({X})}=p\int_0^\infty \lambda^{p-1}
\mu\{x\in X: F^*_{\alpha_1}(x)>\lambda\} \ d\lambda.
$$
Observe that
\begin{align}\label{e2.6}
\left\{x\in X: F^*_{\alpha_1}(x)>\lambda\right\}\subset \left\{x\in X:
\mathcal{M}(\chi_E)(x)>C^{-2}\big(1+\frac{2\alpha_1}{\alpha_2}\big)^{-n}\right\},
\end{align}
where $E:= \{x\in X: F^*_{\alpha_2}(x)>\lambda\}$ and $\mathcal{M}$ denotes the Hardy--Littlewood maximal function.
Indeed, if $F^*_{\alpha_1}(x_0)>\lambda$,  then there exist  $y_0\in X$ and $t_0>0$
such that  $d(x_0,y_0)<\alpha_1t_0$ and $|F(y_0,t_0)|>\lambda$. Hence,
$B(y_0, \alpha_2t_0)\subset E$. It follows that
$$
\frac{\mu(B(x_0, (\alpha_1+\alpha_2)t_0)\cap E)}{V(x_0, (\alpha_1+\alpha_2)t_0)}
\geq \frac{V(y_0, \alpha_2t_0)}{V(x_0, (\alpha_1+\alpha_2)t_0)}.
$$
By (\ref{e1.2}) and (\ref{e1.3}),
\begin{eqnarray*}
 V(x_0, (\alpha_1+\alpha_2)t_0)&\leq& CV(y_0, (\alpha_1+\alpha_2)t_0)
\left(1+\frac{d(x_0,y_0)}{(\alpha_1+\alpha_2)t_0}\right)^n\\
 &\leq& C^2V(y_0, \alpha_2t_0)\left(\frac{\alpha_1+\alpha_2}{\alpha_2}
 \right)^n \left(1+\frac{\alpha_1}{\alpha_1+\alpha_2}\right)^n\\
 &\leq&
 C^2V(y_0, \alpha_2t_0)\left(1+\frac{2\alpha_1}{\alpha_2}\right)^n,
\end{eqnarray*}
which gives
$$
\frac{\mu(B(x_0, (\alpha_1+\alpha_2)t_0)\cap E)}{V(x_0, (\alpha_1+\alpha_2)t_0)}
\geq C^{-2}\left(1+\frac{2\alpha_1}{\alpha_2}\right)^{-n}.
$$
This proves (\ref{e2.6}).   By the weak (1,1) boundedness of Hardy--Littlewood
maximal function, we obtain the proof of Lemma~\ref{le2.2}.
\end{proof}

\begin{prop}\label{prop2.3}
Let $0<p\leq 1$.  Suppose  that an operator $L$ satisfies ${\bf   (H1)}$  and ${\bf   (H2)}$.
Let $\varphi_i \in {\mathscr S}(\mathbb R)$ be even functions with $ \varphi_i(0)=1$
and $\alpha_i>0, i=1,2$.
Then there exists a constant $C=C(n, \varphi_1, \varphi_2, \alpha_1, \alpha_2)$ such that
for every $f\in L^2({X})$, the functions
 $\varphi^{\ast}_{i, L, \alpha}f=\sup\limits_{d(x,y)<\alpha t}|\varphi_i(t\sqrt{L})f(y)|, i=1,2$,
 satisfy
\begin{eqnarray}\label{e2.7}
\left\|\varphi^{\ast}_{1, L, \alpha_1}f\right\|_{L^p({X})}\leq
C\left\|\varphi^{\ast}_{2, L, \alpha_2}f\right\|_{L^p({X})}.
%\Big\|\sup\limits_{|x-y|< \alpha_1 t}|\varphi_1(t^2L)f(y)|\Big\|_{L^1({X})}\leq C
%\Big\|\sup\limits_{|x-y|< \alpha_2 t}|\varphi_2(t^2L)f(y)|\Big\|_{L^1({X})}.
\end{eqnarray}
As a consequence, for any even function $\varphi \in {\mathscr S}(\mathbb R)$  with $\varphi(0)=1$ and $\alpha>0$,
%Then there exists a constant $C=C(n, \varphi, \alpha)$ such that for every $f\in L^1({X})$,
$$
 C^{-1}\|f^{\ast}_L\|_{L^p({X})}\leq  \left\|\varphi^{\ast}_{L, \alpha}f\right\|_{L^p({X})}\leq
 C\|f^{\ast}_L\|_{L^p({X})}.
 $$
\end{prop}

\begin{proof} The argument is similar to that of \cite[Proposition 3.1]{SY} with minor  modifications.
We give a brief argument of this proof for completeness and convenience for the reader.

 For any $0<\alpha_2\leq \alpha_1$, we apply Lemma \ref{le2.2} to have that
$$\left\|\varphi^{\ast}_{L, \alpha_1}f\right\|_{L^p({X})}
\leq C(p,n)\left(1+{2\alpha_1\over\alpha_2}\right)^{n/p}
\left\|\varphi^{\ast}_{L, \alpha_2}f\right\|_{L^p({X})}
$$
for any  $\varphi \in {\mathscr S}(\mathbb R)$.
Now, we let $\psi(x):=\varphi_1(x)-\varphi_2(x)$, and then the proof of \eqref{e4.1} reduces to show that
%It suffices to prove that
\begin{eqnarray}\label{e2.8}
\left\|\psi^{\ast}_{L, 1}f\right\|_{L^p({X})}\leq
C\left\|\varphi^{\ast}_{2, L, 1}f\right\|_{L^p({X})}.
%\Big\|\sup\limits_{|x-y|< t}|\psi(t^2L)f(y)|\Big\|_{L^1({X})}
%\leq C \Big\|\sup\limits_{|x-y|< t}|\varphi(t^2L)f(y)|\Big\|_{L^1({X})}.
\end{eqnarray}

Let us show \eqref{e2.8}.  Let $\varphi\in C^{\infty}_0(\mathbb R)$ be
even, $\mbox{supp}\,\varphi \subset (-1, 1)$. Let $\Phi$ denote the Fourier transform of
$\varphi$   and set $\Psi(x):=x^{2\kappa}\Phi(x)$ and $2\kappa>{(n+1)/p}$.
By the  spectral theory (\cite{Y}), we have
$$
f=C_{\Psi, \  \varphi_2}\int_0^\infty \Psi(s\sqrt{L})\varphi_2(s \sqrt{L})f \, \frac{ds}{s}.
$$
Therefore,
$$
\psi(t\sqrt{L})f(x)=C \int_0^\infty \left(\psi(t\sqrt{L})\Psi(s \sqrt{L})\right)\varphi_2(s \sqrt{L})f(x) \, \frac{ds}{s}.
$$
 Let us denote the kernel of $\psi(t\sqrt{L})\Psi(s\sqrt{L})$ by
 $ K_{\psi(t\sqrt{L})\Psi(s\sqrt{L})}(x,y)$. For  every  $\lambda\in (\frac{n}{p}, \ 2\kappa)$, we write
\begin{eqnarray}\label{e2.9}\\
&&\hspace{-0.5cm}\sup\limits_{d(x,y)< t}|\psi(t\sqrt{L})f(y)| \nonumber\\
&=&C\sup\limits_{d(x,y)<t}\big|\int_0^\infty\!\!\int_X K_{\psi(t\sqrt{L})\Psi(s\sqrt{L})}(y,z)
\varphi_2(s\sqrt{L})f(z) \, d\mu(z)\,\frac{ds}{s}\Big|\nonumber\\
%&\leq C\sup\limits_{d(x,y)< t}\int_0^\infty\!\!\int_X \big|K_{\psi(t\sqrt{L})\Psi(s\sqrt{L})}(y,z)
%\big|\Big(1+\frac{d(x,z)}{s}\Big)^{\lambda} \big|\varphi_2(s\sqrt{L})f(z)
%\big|\Big(1+\frac{d(x,z)}{s}\Big)^{-\lambda} \,d\mu(z)\,\frac{ds}{s}\nonumber\\
&\leq&  \sup\limits_{z,s} \big|\varphi_2(s\sqrt{L})f(z)\big|\Big(1+\frac{d(x,z)}{s}\Big)^{-\lambda}
 \sup\limits_{d(x,y)< t}\int_0^\infty\!\!\int_X \big|K_{\psi(t\sqrt{L})\Psi(s\sqrt{L})}(y,z)
 \big|\Big(1+\frac{d(x,z)}{s}\Big)^{\lambda}\, d\mu(z)\,\frac{ds}{s}.\nonumber
\end{eqnarray}
By (ii) of Lemma~\ref{le2.1}, it follows  that  for  $\eta\in (\lambda, 2\kappa)$,
 \begin{align*}
 \big|K_{\psi(t\sqrt{L})\Psi(s\sqrt{L})}(y,z)\big|\leq C
 \min\left( \left({s\over t}\right)^{2\kappa}, \left({t\over s}\right)^{2}\right)
  {1\over \Big(1+ \frac{d(y,z)}{\max(s,t)}\Big)^{n+\eta}}\frac{1}{V(y,\max(s,t))}.
 \end{align*}
 For any $d(x,y)<t$, one can compute
 \begin{align*}
 \int_{X} \frac{1}{V(y,\max(s,t))} {1\over \Big(1+ \frac{d(y,z)}{\max(s,t)}\Big)^{n+\eta}}
 \Big(1+\frac{d(x,z)}{s}\Big)^{\lambda} \, {d\mu(z)}
 \leq C\max \left(1, \left({t\over s}\right)^{\eta} \right),
\end{align*}
which implies,
\begin{eqnarray*}
 \int_{{X}} \big|K_{\psi(t\sqrt{L})
\Psi(s\sqrt{L})}(y,z)\big|\left(1+\frac{d(x,z)}{s}\right)^{\lambda} \, d\mu(z)
&\leq&   C  \min\left( \left({s\over t}\right)^{2\kappa}, \left({t\over s}\right)^{2}\right)
  \max \left(1, \left({t\over s}\right)^{\eta} \right)\\
 &\leq&   C \min\left( \left({s\over t}\right)^{2\kappa-\eta},  \left({t\over s}\right)^{2}\right),
\end{eqnarray*}
for any $d(x,y)<t$. Hence
\begin{eqnarray*}
 \sup\limits_{d(x,y)< t}\int_0^\infty\!\!\int_{{X}} \big|K_{\psi(t\sqrt{L})
\Psi(s\sqrt{L})}(y,z)\big|\left(1+\frac{d(x,z)}{s}\right)^{\lambda} \, {d\mu(z)}\,\frac{ds}{s}
 \leq C\int_0^\infty \min\left( \left({s\over t}\right)^{2\kappa-\eta},  \left({t\over s}\right)^{2}\right) \frac{ds}{s}\leq C.
\end{eqnarray*}
This, in combination with
\eqref{e2.9}  and the condition $\lambda\in (\frac{n}{p}, \ 2\kappa)$, implies
\begin{eqnarray*}
\left\|\psi^{\ast}_{L, 1}f\right\|_{L^p({X})}=\Big\|\sup\limits_{d(x,y)< t}|\psi(t\sqrt{L})f(y)|\Big\|_{L^p_x({X})}&\leq&
C\Big\|\sup\limits_{z,s} \big|\varphi_2(s\sqrt{L})f(z)\big|\Big(1+\frac{d(x,z)}{s}
\Big)^{-\lambda}\Big\|_{L^p_x({X})}\nonumber\\
&\leq& C\left\|\sup\limits_{d(x,y)<  t}|\varphi_2(t\sqrt{L})f(y)|\right\|_{L^p_x({X})}\\
&= &C\Big\|\varphi^{\ast}_{2, L, 1}f\Big\|_{L^p({X})},
\end{eqnarray*}
where  the second inequality above  can be proved easily by  combining Lemma \ref{le2.2} and
 the  argument of \cite[Theorem 2.4]{CT}.
This completes  the proof of Proposition~\ref{prop2.3}.
\end{proof}

\medskip

\section{Equivalence of Hardy spaces $H^p_{L,  {max}}(X) $ and $ H^p_{L, {rad}}({X})$}
\setcounter{equation}{0}

Assume that the metric measure space  $X$ satisfies  the
doubling conditions~\eqref{e1.1} and \eqref{e1.2} with exponent $n$.
In this section we  will show  (i) of Theorem~\ref{th1.1} to obtain an equivalence
of Hardy spaces $H^p_{L,  {  rad}}(X) $ and $ H^p_{L, {  max}}({X})$.
By \eqref{e1.4} and \eqref{e1.7}, we have
 \begin{eqnarray}\label{ed}
 H^p_{L, max}({X})\subseteq
H^p_{L, rad }({X}).
\end{eqnarray}
Now for every even functions $\varphi \in {\mathscr S}(\mathbb R)$ with $ \varphi(0)=1$, and
for every $f\in L^2({X})$ we define
$$\varphi^{\ast}_{L}(f)(x)=\sup\limits_{d(x,y)<  t}|\varphi(t\sqrt{L})f(y)|
$$
and
$$\varphi^{+}_{L}(f)(x)=\sup\limits_{t>0}|\varphi(t\sqrt{L})f(x)|
$$
Then we have  the the following result.

 \begin{theorem}\label{th3.1}
Let $(X,d,\mu)$ be as in \eqref{e1.1} and \eqref{e1.2}.
 Suppose  that an operator $L$ satisfies ${\bf   (H1)}$  and ${\bf   (H2)}$.
Fix $0<p\leq 1$. Then there exists a constant  $C=C(p,\varphi)>0$  such that for every $f\in L^2(X)$,
\begin{align}\label{e3.1}
\|\varphi^*_L(f)\|_{L^p(X)} \leq C \|\varphi^+_{L}(f)\|_{L^p(X)},
\end{align}
 and hence by \eqref{ed}
$$
 H^p_{L, {  max}}({X}) \simeq  H^p_{L, { rad}}({X})  .
$$
\end{theorem}

\begin{proof}
\iffalse
 Let $\varphi\in{\mathscr S}(\mathbb{R})$ be an even  function with $ \varphi(0)=1$.
To prove Theorem~\ref{th3.1}, we will show that for every $f\in L^2(X)$,
\begin{align}\label{e3.1}
\|\varphi^*_L(f)\|_{L^p(X)} \leq C \|\varphi^+_{L}(f)\|_{L^p(X)} \ \ \ \ 0<p\leq 1,
\end{align}
with some constant  $C=C(p,\varphi)>0$   independent of $f$.\fi

For every $N>0$, we define
$$
M_{L,\varphi,N}^{**}(f)(x):= \sup\limits_{y\in X, s>0} \frac{|\varphi(s\sqrt{L})f|}{\left(1+\frac{d(x,y)}{s}\right)^N}.
$$
By \eqref{e1.4}, we have
\begin{align}\label{e3.2}
\varphi^*_L(f)(x)\leq 2^N M_{L,\varphi,N}^{**}(f)(x).
\end{align}
We now claim that if $0<\theta<1$ and $N\theta>2n$, then there exists $C=C(p,\varphi,N, \theta)>0$
such that for every $f\in L^2(X)$,
\begin{align}\label{e3.3}
M_{L,\varphi,N}^{**}(f)(x)\leq C \Big[\mathcal{M}\big((\varphi^+_{L}(f))^\theta\big)(x)\Big]^{1/\theta}
  \ {\rm  a.e.}\  x\in X.
\end{align}
If the claim is proved, then we can choose $N=2(n+1)/p$ and $\theta= \frac{(2n+1)p}{2(n+1)}$
and apply the $L^r (r>1)$ boundedness of  Hardy--Littlewood maximal operator to obtain that
for any $f\in L^2(X)$
\begin{align*}
\|M_{L,\varphi,N}^{**}(f)\|_{L^p(X)}\leq C
\left\|\Big[\mathcal{M}\big((\varphi^+_{L}(f))^\theta\big)(x)\Big]^{1/\theta}\right\|_{L^p(X)}
\leq C \|\varphi^+_{L}(f)\|_{L^p(X)},
\end{align*}
which, together with (\ref{e3.2}), yields  (\ref{e3.1}).

It remains to prove (\ref{e3.3}).  Let $\varphi\in C^{\infty}_0(\mathbb R)$ be
even, $\mbox{supp}\,\varphi \subset (-1, 1)$. Let $\Phi$ denote the Fourier transform of
$\varphi$   and set $\Psi(x):=x^{2\kappa}\Phi(x)$,
$x\in{\mathbb{R}}$ and $\kappa>N/2$.  For every $f\in L^2({X})$ one can write
 \begin{align}\label{e3.4}
f
%&=c_{\Psi,\varphi}\int_0^\infty\Psi({t\sqrt{L}})\varphi(t\sqrt{L})f \, \frac{{\rm d}t}{t}\nonumber\\[4pt]
&=\lim_{\epsilon\to 0}c_{\Psi,\varphi}\int_{\epsilon}^{1/\epsilon}
\Psi({t\sqrt{L}}) \varphi(t\sqrt{L})f \, \frac{{\rm d}t}{t}
\end{align}
with the integral converging in $L^2({X}).$

Set
\begin{align*}
\eta(x)
:= c_{\Psi,\varphi} \int_1^{\infty} \Psi(tx) \varphi(tx) \frac{dt}{t}
=
c_{\Psi,\varphi} \int_x^{\infty} \Psi(y)  \varphi(y) \frac{{dy}}{y}, \quad x\neq 0
\end{align*}
with $\eta(0)=1$. It follows that  $\eta\in{\mathscr S}({\mathbb R})$ is an even function.
By the spectral theory (\cite{Yo}) again, one can write, for any $s>0$,
\begin{align}\label{e3.5}
\eta(s\sqrt{L})f=c_{\Psi,\varphi}\int_s^\infty
\Psi({t\sqrt{L}})\varphi(t\sqrt{L})f \, \frac{{\rm d}t}{t},
\end{align}
which, together with (\ref{e3.4}), yields that for any $f\in L^2(X)$,
\begin{align}\label{e3.6}
f&=\eta(s\sqrt{L})f+c_{\Psi,\varphi}\int_0^s
\Psi({t\sqrt{L}})\varphi(t\sqrt{L})f \, \frac{{\rm d}t}{t}.
\end{align}
Let $0<\theta<1$, $N\theta>2n$. By (\ref{e3.6}),  there holds

\begin{align*}
\frac{|\varphi(s\sqrt{L})f(y)|}{\big(1+\frac{d(x,y)}{s}\big)^N}&\leq \frac{\big|\eta(s\sqrt{L})
\varphi(s\sqrt{L})f(y)\big|}{\big(1+\frac{d(x,y)}{s}\big)^N}
+\frac{c_{\Psi,\varphi}}{\big(1+\frac{d(x,y)}{s}\big)^N}\Big|\int_0^s
{\varphi(s\sqrt{L})\Psi({t\sqrt{L}})}\varphi(t\sqrt{L})f(y) \, \frac{{\rm d}t}{t}\Big|\\
&=:{ I}+ { II}.
\end{align*}

Now we  apply an argument of Str\"omberg and Torchinsky  as   \cite[Chapter V, Theorem 5]{ST}
%(see Lemma 3.2, \cite{NS})
on page 64.  For the term $I$, we use (i) of Lemma \ref{le2.1} to obtain

 \begin{align*}
{I}&\leq \frac{C}{\big(1+\frac{d(x,y)}{s}\big)^N}\, \int_X \frac{1}{V(z,s)}
  {1\over \left(1+ \frac{d(y,z)}{s}\right)^{N}} \big|\varphi(s\sqrt{L})f(z)\big|\, {d\mu(z)}\\
 &\leq C\int_X \frac{1}{V(z,s)}
  {1\over \left(1+ \frac{d(x,z)}{s}\right)^{N}} \big|\varphi(s\sqrt{L})f(z)\big|\, {d\mu(z)}\\
 &\leq  C\int_X \frac{1}{V(x,s)} \frac{\big|\varphi(s\sqrt{L})f(z)\big|^{\theta}}
  {\left(1+ \frac{d(x,z)}{s}\right)^{N\theta-n}} \, {d\mu(z)} \ \   \Big(M_{L,\varphi,N}^{**}(f)(x)\Big)^{1-\theta},
\end{align*}
where in the last inequality above we have used (\ref{e1.3}). By noting that $N\theta >2n$, we obtain,

\begin{align}\label{e3.7}
{ I}\leq C \mathcal{M}\Big(\big|\varphi_{L}^+f\big|^{\theta}\Big)(x)
\Big(M_{L,\varphi,N}^{**}(f)(x)\Big)^{1-\theta}.
\end{align}

\medskip

Let us estimate the term ${II}$.
One writes
$
\Psi(tx)\psi(sx)=
 ({t\over s})^{2\kappa} [\Phi(tx)(sx)^{2\kappa}\psi(sx)].
$
We then apply an argument as in (ii)  of Lemma \ref{le2.1} to know that
the kernel $K_{\Psi({t\sqrt{L}})\varphi(s\sqrt{L}) }(x,y)$
 of $ \Psi({t\sqrt{L}})\varphi(s\sqrt{L})$
satisfies
\begin{eqnarray*}
\big|K_{\Psi({t\sqrt{L}})\varphi(s\sqrt{L})}(y,z)\big|
\leq  C\, \Big(\frac{t}{s}\Big)^{2\kappa} \frac{1}{V(z,s)}
  {1\over \left(1+ \frac{d(y,z)}{s}\right)^{N}}
\end{eqnarray*}
for all $s>t>0$ and $y,z\in {X}$. This yields
\begin{align*}
{II}&\leq \frac{C}{\big(1+\frac{d(x,y)}{s}\big)^N}\, \int_0^s\!\!\int_X \Big(\frac{t}{s}\Big)^{2\kappa} \frac{1}{V(z,s)}
  {1\over \left(1+ \frac{d(y,z)}{s}\right)^{N}} \big|\varphi(t\sqrt{L})f(z)\big| {d\mu(z)} \frac{dt}{t}\\
  &\leq C\, \int_0^s\!\!\int_X \Big(\frac{t}{s}\Big)^{2\kappa} \frac{1}{V(z,s)}
  {1\over \left(1+ \frac{d(x,z)}{s}\right)^{N}} \big|\varphi(t\sqrt{L})f(z)\big| {d\mu(z)} \frac{dt}{t}\\
  &\leq C \int_0^s\!\!\int_X \Big(\frac{t}{s}\Big)^{2\kappa-N} \frac{1}{V(z,t)}
  \frac{\big|\varphi^+_{L}f(z)\big|^\theta}{\left(1+ \frac{d(x,z)}{t}\right)^{N\theta}}
   {d\mu(z)} \frac{dt}{t}\ \  \Big(M_{L,\varphi,N}^{**}(f)(x)\Big)^{1-\theta}\\
     &\leq C \int_0^s\!\!\int_X \Big(\frac{t}{s}\Big)^{2\kappa-N} \frac{1}{V(x,t)}
  \frac{\big|\varphi^+_{L}f(z)\big|^\theta}{\left(1+ \frac{d(x,z)}{t}\right)^{N\theta-n}}
   {d\mu(z)} \frac{dt}{t}\ \  \Big(M_{L,\varphi,N}^{**}(f)(x)\Big)^{1-\theta},
\end{align*}
where in the last inequality above we have used (\ref{e1.3}). Since $2\kappa>N$ and $N\theta>2n$, we have
\begin{align}\label{e3.8}
{II}  &\leq C \int_0^s\!\! \Big(\frac{t}{s}\Big)^{2\kappa-N} \frac{dt}{t}
  \mathcal{M}\Big({ |\varphi^+_{L}f |^\theta}\Big)(x) \ \   \Big(M_{L,\varphi,N}^{**}(f)(x)\Big)^{1-\theta}\\
  &\leq C \mathcal{M}\Big({ |\varphi^+_{L}f |^\theta}\Big)(x) \ \   \Big(M_{L,\varphi,N}^{**}(f)(x)\Big)^{1-\theta}. \nonumber
\end{align}

Combining \eqref{e3.7} and \eqref{e3.8}, we have proved that
\begin{align}\label{e3.9}
M_{L,\varphi,N}^{**}(f)(x)\leq C\mathcal{M}\Big({ |\varphi^+_{L}f(z) |^\theta}\Big)(x) \ \
 \Big(M_{L,\varphi,N}^{**}(f)(x)\Big)^{1-\theta}.
\end{align}
Finally, it can be verified that  for any $f\in L^2(X)$, $M_{L,\varphi,N}^{**}(f)(x)<\infty$, for  $a.e. \ x\in X$.
From \eqref{e3.9},      \eqref{e3.3} follows readily.
The proof of Theorem~\ref{th3.1} is complete.
\end{proof}

\medskip

\section{Proof of Theorem~\ref{th1.1}}
\setcounter{equation}{0}

In this section we   continue to  show (ii) of Theorem~\ref{th1.1} to give a
$(p, \infty, M)$-atomic representation  for the Hardy spaces
 $ H^p_{L,max}({X})$.
To do it, we first recall  the following Whitney type covering lemma on space of homogeneous type $X$.

\begin{lemma}\label{le4.1}
Suppose that $O\subseteq X$ is a open set with finite measure.
There exists a sequence of points $\{\xi_k\}_{k=1}^\infty\in O$ and a collection of balls $B(\xi_k, \rho_k)$
where $\rho_k:=d(\xi_k,O^c)$ such that

\medskip

i) $\bigcup_kB(\xi_k, \rho_k/2)=O$;

\medskip

ii) $\{B(\xi_k, \rho_k/10)\}_{k=1}^\infty$ are disjoint.

\end{lemma}
\begin{proof}
The proof of this lemma is essentially given in \cite[Chapter III, Theorem 1.3]{CW}
and is omitted. See also \cite{DKKP, MS}.
\end{proof}

\noindent
{\it Proof of (ii) of Theorem~\ref{th1.1}}.
It suffices to show that for
 $f\in H^p_{L,max}(X) \cap L^2({X}),$    $f$ has a $(p, \infty, M)$
atomic representation.

We start with a suitable
version of the Calder\'on reproducing formula. Let $\varphi\in C^{\infty}_0(\mathbb R)$ be an even function with
 $\mbox{supp}\,\varphi \subset (-1, 1)$. Let $\Phi$ denote the Fourier transform of
$\varphi$,
  and set $\Psi(x):=x^{2M}\Phi(x)$,
$x\in{\mathbb{R}}$. By the spectral theory (\cite{Yo}), for
every $f\in L^2({X})$ one can write
 \begin{align}\label{e4.1}
f
%&=c_{\Psi}\int_0^\infty\Psi({t\sqrt{L}})t^2Le^{-t^2L}f \, \frac{{\rm d}t}{t}\nonumber\\[4pt]
&=\lim_{\epsilon\to 0}c_{\Psi}\int_{\epsilon}^{1/\epsilon}
\Psi({t\sqrt{L}})t^2L e^{-t^2L}f \, \frac{{\rm d}t}{t}
\end{align}
with the integral converging in $L^2({X}).$

Set
\begin{align*}
\eta(x)
:= c_{\Psi} \int_1^{\infty} t^2 x^2\Psi(tx) e^{-t^2x^2} \frac{dt}{t}
=
c_{\Psi} \int_x^{\infty} y\Psi(y)  e^{-y^2} {dy}, \quad x\neq 0
\end{align*}
with $\eta(0)=1$.  Then  $\eta\in{\mathscr S}({\mathbb R})$ is an even function.
%$$
%\eta(a  x)-\eta(b  x)=c_{\Psi}\int_a^b t^2x^2\Psi(t  x)e^{-t^2 x^2} \frac{dt}{t}.
%$$
By the spectral theory (\cite{Yo}) again, one has
\begin{align}\label{e4.2}
c_{\Psi}\int_a^b
\Psi({t\sqrt{L}})t^2Le^{-t^2L}f \, \frac{{\rm d}t}{t}=\eta(a\sqrt{L})f(x)-\eta(b\sqrt{L})f(x).
\end{align}
Define,
$${\mathscr M}_Lf(x):=\sup\limits_{d(x,y)<5t}\Big(|t^2Le^{-t^2L}f(y)|+|\eta(t\sqrt{L})f(y)|\Big).
$$
 By Proposition~\ref{prop2.3},  it follows that
$$\|{\mathscr M}_Lf\|_{L^p({X})}\leq C\|f\|_{\HMAX}, \ \ \ \ \  0<p\leq 1.
$$

 In the sequel, if $O$ is an
open subset of ${\mathbb R}^n$, then the ``tent" over $O$, denoted
by ${\widehat O}$, is given as ${\widehat O}:=\{(x,t)\in X\times (0,+\infty):  B(x, 4t)\subset O\}$.
For $ i\in{\Bbb Z}$, we  define the family of sets $O_i: =\{x\in {X}: {\mathscr M}_Lf(x)>2^i\}$.
We obtain a decomposition for $X\times (0,+\infty)$ as follows:
\begin{align}\label{g}
X\times (0,+\infty)&=\bigcup_{i\in \mathbb{Z}} \widehat{O}_{i}\nonumber\\
&=\bigcup_{i\in \mathbb{Z}}\big(\widehat{O_i}
\big\backslash \widehat{O_{i+1}}\big)\nonumber\\
&=:\bigcup_{i\in \mathbb{Z}} T_{i},
\end{align}
where
\begin{align*}
T_{i}:=\widehat{O_i}
\big\backslash {\widehat{O_{i+\!1}}}.
\end{align*}
% Using the formula (\ref{e3.1}), one can write
% \begin{align}\label{e3.3}
%f=\sum_{ i }c_{\Psi}\int_0^\infty
%\Psi(t\sqrt{L})\Big( \chi_{T_i}t^2Le^{-t^2L}f\Big)\, \frac{{\rm d}t}{t}
%\end{align}
%with the sum converging in $L^2({X})$.

Note that for each $i\in \mathbb{Z}$,  $O_i$ is open set with $\mu(O_i)<\infty$.
By Lemma \ref{le4.1}, we can further  decomposition $O_i$ into ``balls" of X.
More precisely,  for each $i\in \mathbb{Z}$, there exists a sequence of points
$\{\xi^k_i\}_{k=1}^\infty\in  O_i$, such that

1) \ $O_i=\cup^\infty_{k=1} B^k_i$;

2) \ $\{\frac{1}{5}B^k_i\}_{k=1}^\infty$ are disjoint, where  $B^k_i:=B(\xi_i^k, \rho_i^k/2)$
 and  $\rho_i^k:=d(\xi_i^k,{O_i}^c)$.

\noindent
For any $E\subset X$, we define  the ``cone" of $E$ by
\begin{align}\label{definition of cone}
{\mathcal R}(E):=\{(y,t): d(y,E)<2t\}.
\end{align}
For every $k=0, 1, 2, \cdots, $ we set
\begin{eqnarray}\label{gg}
{\mathcal R}(B^0_i):=\emptyset,\ \ \
T_i^k:=T_i\cap  \big({\mathcal R}(B^k_i)\backslash \cup_{j=0}^{k-1}{\mathcal R}(B^j_i)\big).
\end{eqnarray}
 It is easy to see that
$
\widehat{O_i}\subset  \cup_{j\in \mathbb{N}} {\mathcal R}(B^j_i)$ and $
 T_i^k\cap T_{i'}^{k'}=\emptyset$ if $k\neq k'$ or $ i\neq i'$.
By \eqref{g} and \eqref{gg}, we can obtain a further decomposition for $X\times (0,+\infty)$ as follows:
 \begin{align}\label{eoo}
X\times (0,+\infty)
 &=\bigcup_{i\in \mathbb{Z}}\bigcup_{k\in \mathbb{N}}\Big( T_{i} \cap {\mathcal R}(B^k_i)\Big)\nonumber\\
 &=\bigcup_{i\in \mathbb{Z}}\bigcup_{k\in \mathbb{N}}\left(
 T_i\cap  \big({\mathcal R}(B^k_i)\backslash \cup_{j=0}^{k-1}{\mathcal R}(B^j_i)\big)\right)\nonumber\\
&= \bigcup_{i\in \mathbb{Z}}\bigcup_{k\in \mathbb{N}} T_i^k.
\end{align}
By (\ref{e4.1}), this leads us to  write
 \begin{align}\label{e4.3}
f&=\sum_{ i\in \mathbb{Z},k\in \mathbb{N}}c_{\Psi}\int_0^\infty
\Psi(t\sqrt{L})\Big( \chi_{T_i^k}t^2Le^{-t^2L}f\Big)\, \frac{{\rm d}t}{t}\nonumber\\[4pt]
&= : \sum_{i\in \mathbb{Z},k\in \mathbb{N}}  \lambda_{i}^k a^k_{i},
\end{align}
where $\lambda_{i}^k:= 2^{i}{\mu(B_{i}^k)}^{1/p}, \    a_{i}^k:=L^Mb_{i}^k$, and
$$
b_{i}^k:=(\lambda_{i}^k)^{-1} c_{\Psi}\int_0^\infty t^{2M} \Phi
(t\sqrt{L})\Big( \chi_{T_{i}^k}t^2Le^{-t^2L}f\Big)\,
\frac{{\rm d}t}{t}.
$$
We see that the sum (\ref{e4.3}) converges in $L^2(X)$. Indeed, since
for each $f\in L^2(X)$,
$$\left(\int_{\UHRN} |t^2Le^{-t\sqrt{L}}f(y)|^2 \, \frac{{d\mu(y)}dt}{t}\right)^{1/2}\leq C \|f\|_{L^2(X)}.
$$
By (\ref{e4.3}),
 \begin{eqnarray*}
\left\|\sum_{|i|>N_1, k>N_2}  \lambda_{i}^k a_{i}^k\right\|_{L^2(X)}&=&c_{\Psi}\left\|\sum_{|i|>N_1, k>N_2}\int_{\UHRN}K_{
(t^2L)^{M }\Phi({t\sqrt{L}})}(x,y) \chi_{T_{i}^k}(y,t) t^2Le^{-t\sqrt{L}}f(y) \, \frac{{d\mu(y)}dt}{t}\right\|_{L^2(X)} \\
&\leq& \sup\limits_{\|g\|_2\leq 1} \sum_{|i|>N_1, k>N_2}\int_{T_{i}^k}
\big|(t^2L)^{M }\Phi({t\sqrt{L}})g(y) t^2Le^{-t\sqrt{L}}f(y)\big| \, \frac{{d\mu(y)}dt}{t}\\
&\leq &C\left(\sum_{|i|>N_1, k>N_2}\int_{T_{i}^k}  |t^2Le^{-t\sqrt{L}}f(y)|^2 \, \frac{{d\mu(y)}dt}{t}\right)^{1/2} \to 0
\end{eqnarray*}
as $N_1\to \infty, N_2\to \infty.$

Next, we will show that, up to a normalization by a  multiplicative constant,
the $a_{i}^k$ are $(p, \infty, M)$-atoms. Once the claim is established, we
  have
\begin{align*}
\sum_{i\in \mathbb{Z},k\in \mathbb{N}}|\lambda_{i}^k|^p
 =  \sum_{i\in \mathbb{Z},k\in \mathbb{N}}2^{ip}\mu(B_i^k)
&\leq 5^n \sum_{i\in \mathbb{Z},k\in \mathbb{N}}2^{ip}\mu(\frac{1}{5}B_i^k)\\
&\leq C\sum_{i\in \mathbb{Z}}2^{ip}\mu(O_i)\\
&\leq C\|f\|^p_{\HMAX}
\end{align*}
as desired.

Let us now prove   that for every $i\in \mathbb{Z}$ and $k\in \mathbb{N}$, the function
$C^{-1} a_{i}^k$
 is a $(p, \infty, M)$-atom associated with the ball $B(\xi_i^k,5\rho_i^k)$ for some  constant $C$.
Observe that if $(y,t)\in T_{i}^k$, then $B(y,4t)\in O_i$.    It implies that
$4t\leq  d(y,(O_i)^c)$. Note that $d(y, B_i^k)<2t$, then $d(y, (O_i)^c)\leq  d(y, B_i^k)+ 2\rho_i^k<2t+2\rho_i^k$.
Hence, we have that $t<\rho_i^k$.
From the formula \eqref{e2.3},   it is easy to see that the integral kernel
 $K_{(t^{2}L)^k\Phi(t\sqrt{L})}$ of the operator $(t^{2}L)^k\Phi(t\sqrt{L})$ satisfies
$$
{\rm supp}\, K_{(t^{2}L)^k\Phi(t\sqrt{L})} \subseteq
\big\{(x,y)\in\mathbb{R}^n\times\mathbb{R}^n: |x-y|\leq t\big\}.
$$
\noindent It   follows that $d(x, B_i^k)\leq d(x,y)+d(y, B_i^k)<3t\leq 3\rho_i^k$ and  $d(x, \xi_i^k)<4\rho_i^k$.
Hence,  for every $j=0,1, \cdots, M$
 $$
{\rm supp}\, \big(L^j b_{i}^k\big) \subseteq B(\xi_i^k,4\rho_i^k).
$$
It remains to show that   $\big\|(\rho_i^k)^2L)^{j}b^k_{i}\big\|_{L^\infty(X)}\leq C (\rho_i^k )^{2M}
\mu(B_{i}^k)^{-1/p},\ j=0,1, \cdots, M$.

In this case $j=0,1,\cdots, M-1$,
  it reduces to show
\begin{align}\label{e4.4}
\Big|\int_0^\infty\int_{\R} K_{
t^{2M}L^j\Phi(t\sqrt{L})}(x,y) \chi_{T_{i}^k}(y,t)t^2Le^{-t^2L}f(y) \,{d\mu(y)}
\frac{{\rm d}t}{t}\Big|\leq C2^i (\rho_i^k)^{2(M-j)}.
\end{align}
Indeed,  if $\chi_{T_{i}^k}(y,t)=1$,
then $(y,t)\in (\widehat{O_{i+1}})^c$, and so $B(y,4t)\cap (O_{i+1})^c\neq \emptyset$.
Let  $\bar{x}\in B(y,4t)\cap (O_{i+1})^c$. We have that $|t^2Le^{-t^2L}f(y)|\leq
{\mathscr M}_{L}f(\bar{x})\le 2^{i+1}$. By (i) of Lemma~\ref{le2.1},
\begin{align*}
&\Big|\int_0^\infty\int_{\R} K_{t^{2M}L^j\Phi(t\sqrt{L})}(x,y) \chi_{T_{i}^k}(y,t)
t^2Le^{-t^2L}f(y) \,{d\mu(y)}\frac{{\rm d}t}{t}\Big|\\
&\leq C2^i \Big|\int_0^{\rho_i^k} t^{2(M-j)}\int_{\R}
\big|K_{(t^2L)^j\Phi(t\sqrt{L})}(x,y)\big|  \,{d\mu(y)} \frac{{\rm d}t}{t}\Big|\\
&\leq C2^i \int_0^{\rho_i^k} t^{2(M-j)} \frac{{\rm d}t}{t}\\
&\leq C2^i(\rho_i^k)^{2(M-j)}
\end{align*}
since   $j=0,1,\cdots, M-1$.

Next, let us consider this case $j=M$. we will show  that for every $i\in \mathbb{Z},k\in \mathbb{N}$,
\begin{align}\label{e4.5}
\Big|\int_0^\infty\!\!\int_{X}
K_{\Psi(t\sqrt{L})}(x,y) \chi_{T_{i}^k}(y,t)t^2Le^{-t^2L}f(y) \,{d\mu(y)}\frac{{\rm d}t}{t}\Big|\leq C2^{i}.
\end{align}
To prove  (\ref{e4.5}), we need the following result, which plays a crucial
role in the proof of (ii) of Theorem~\ref{th1.1}.

\begin{lemma}\label{le4.2}
  Fix $x\in O_i$, the properties of the set defining  $\chi_{T_{i}^k}(y,t)$ imply that
  there exist intervals $(0,b_0),$ $(a_1,b_1), \cdots , (a_N, +\infty)$, where  $0<b_0\leq
a_1<b_1\leq \cdots \leq a_N<+\infty$, $1\leq N\leq 4$ such that  for $j=0,1, \cdots, N-1$,
there hold   $a_{j+1}\leq 3^4b_j$   and

\smallskip

\begin{itemize}
\item[(a)] $K_{\Psi(t\sqrt{L})}(x,y) \chi_{T_{i}^k}(y,t)=0$ for $t>a_N$;

\smallskip

\item[(b)]    either $K_{\Psi(t\sqrt{L})}(x,y) \chi_{T_{i}^k}(y,t)=0$   or $K_{\Psi(t\sqrt{L})}(x,y)
\chi_{T_{i}^k}(y,t)=K_{\Psi(t\sqrt{L})}(x,y)$  for all $t\in (a_j, b_j)$;

\item[(c)]    either $K_{\Psi(t\sqrt{L})}(x,y) \chi_{T_{i}^k}(y,t)=0$   or $K_{\Psi(t\sqrt{L})}(x,y)
\chi_{T_{i}^k}(y,t)=K_{\Psi(t\sqrt{L})}(x,y)$  for all  $t\in (0,b_0)$.
\end{itemize}
\end{lemma}
\begin{proof}
Recall that for any set $E\subset X$,  ${\mathcal R}(E)$ is given in \eqref{definition of cone}.  It is easy to see that for every set $E_1 \subset X$ and $E_2\subset X$,
there holds
$
{\mathcal R}(E_1) \cup {\mathcal R}(E_2)={\mathcal R}(E_1\cup E_2).
$
 One can write
\begin{eqnarray*}
{\mathcal R}(B^k_i)\backslash \bigcup_{j=0}^{k-1}{\mathcal R}(B^j_i)
&=&\bigcup_{j=0}^{k}{\mathcal R}(B^j_i)\backslash \bigcup_{j=0}^{k-1}{\mathcal R}(B^j_i)\\[2pt]
&=&{\mathcal R}\left(\bigcup_{j=0}^{k}B^j_i\right) \backslash {\mathcal R}\left(\bigcup_{j=0}^{k-1}B^j_i\right)\\[2pt]
&:=&{\mathcal R} (E_i^k)\backslash {\mathcal R} (E_i^{k-1}),
\end{eqnarray*}
which, in combination with
 $T_i^k=T_i\cap ({\mathcal R}(B^k_i)\backslash \bigcup_{j=1}^{k-1}{\mathcal R}(B^j_i))$(see \eqref{gg}),
  gives

\begin{align}\label{e4.6}
\chi_{T_{i}^k}(y,t)&=\chi_{\widehat{O_i}}(y,t)\cdot\chi_{(\widehat{O_{i+\!1}})^c}(y,t)
\cdot \chi_{{\mathcal R}(E_i^k)}(y,t) \cdot \chi_{({\mathcal R}(E_i^{k-1}))^c}(y,t)\nonumber\\
&=:\prod_{\ell=1}^4 \chi_\ell(y,t).
\end{align}
From \eqref{e4.6}, we know  that if $\chi_{T_{i}^k}(y,t)=1$, then
$\chi_\ell(y,t)=1 $ for all $\ell=1, 2, 3, 4$. That is,   if either of $\chi_\ell(y,t)=0$, then $\chi_{T_{i}^k}(y,t)=0$.

 To prove Lemma~\ref{le4.2}, we claim  that for $\ell=1,2,3,4$,  there exist  numbers $b^{(\ell)}$ and
$a^{(\ell+1)}$, $0<b^{(\ell)}\leq a^{(\ell+1)}, \ a^{(\ell+1)}\leq 3b^{(\ell)}$  such that
 either $K_{\Psi(t\sqrt{L})}(x,y) \chi_{\ell}(y,t)=0$   or
 $K_{\Psi(t\sqrt{L})}(x,y) \chi_{\ell}(y,t)=K_{\Psi(t\sqrt{L})}(x,y)$  for all $t$ in each
 of the intervals complementary to $(b^{(\ell)}, a^{(\ell+1)})$. And for at least one of
 $\chi_\ell(y,t)$,    $K_{\Psi(t\sqrt{L})}(x,y) \chi_{\ell}(y,t)=0$ for $\ell> a^{(\ell+1)}$.
 By \eqref{e4.6}, we see that
$$
K_{\Psi(t\sqrt{L})}(x,y)\chi_{T_{ij}}(y,t)= \prod_{\ell=1}^{4}\chi_\ell(y,t)
K_{\Psi(t\sqrt{L})}(x,y)
$$
 equals   $K_{\Psi(t\sqrt{L})}(x,y)$ or $0$ when $t$ is in  each of the intervals complementary
to  $\cup_{\ell=1}^{4} (b^{(\ell)}, a^{(\ell+1)})$. From this,   Lemma~\ref{le4.2} follows readily.

We now  prove our claim. Fix $x\in O_i$ and $d(x,y)<t$.  Let us consider the following four cases.

\medskip

\noindent
{\bf Case 1:    $\chi_{1}(y,t)=\chi_{\widehat{O_i}}(y,t)=1$}.

\smallskip

 In this case, we  choose $b^{(1)}= {1\over 5}d(x, {O_i}^c)$ and $a^{(2)}= {1\over 2}d(x, {O_i}^c)$,
 and so  $a^{(2)}\leq 3 b^{(1)}$.
  If $t<b^{(1)}$,  then
$
d(y, {{O_i}^c})\geq d(x, {{O_i}^c})-d(x,y)> 5t-t=4t.
$
This tells us
$$K_{\Psi(t\sqrt{L})}(x,y) \chi_{\widehat{O_i}}(y,t)=K_{\Psi(t\sqrt{L})}(x,y),  \quad  {\rm for}  \ \    t< b^{(1)}.
$$
On the other hand, if $t>a^{(2)}$, then
$
d(y, {{O_i}^c})\leq  d(x,{{O_i}^c})+d(x,y)<4t.
$
From this, we have
$$K_{\Psi(t\sqrt{L})}(x,y) \chi_{\widehat{O_i}}(y,t)=0,  \quad  {\rm for},  \ \  t>a^{(2)}.
$$

\medskip

\noindent
{\bf Case 2:   $\chi_{2}(y,t)=\chi_{(\widehat{O_{i+\!1}})^c}(y,t)=1$}.

\smallskip

In this case, we consider two cases: $d(x,{O_{i+1}}^c)=0$ and $d(x,{O_{i+1}}^c)>0$.

\medskip

\noindent
{\bf Subcase 2.1:} \ $d(x,{O_{i+1}}^c)=0$.

\smallskip

It follows that
$
d(y, {{O_{i+1}}^c})\leq  d(x,{{O_{i+1}}^c})+d(x,y)<t<4t.
$
Hence,
$$K_{\Psi(t\sqrt{L})}(x,y) \chi_{(\widehat{O_{i+\!1}})^c}(y,t)=K_{\Psi(t\sqrt{L})}(x,y),   \quad   {\rm for} \  \  t>0.
$$
So we can choose  $b^{(2)}$ and $a^{(3)}$ to be any positive number. For example, we
let  $b^{(2)}=b^{(1)}$ and $a^{(3)}=a^{(2)}$.

\smallskip

\noindent
{\bf Subcase 2.2:}   $d(x,{O_{i+1}}^c)>0$.

\smallskip
Let us choose  $b^{(2)}=\frac{1}{5}d(x, {{O_{i+1}}^c})$
and $a^{(3)}=\frac{1}{2}d(x, {{O_{i+1}}^c})$.
 If  $t<b^{(2)}$, then
$
d(y, {{O_{i+1}}^c})\geq d(x,{{O_{i+1}}^c})-d(x,y)> 5t-t=4t.
$
which gives
$$K_{\Psi(t\sqrt{L})}(x,y) \chi_{(\widehat{O_{i+\!1}})^c}(y,t)=0  \quad {\rm for}   \ \  t< b^{(2)},
$$
If $t> a^{(3)}$, then
$
d(y, {{O_{i+1}}^c})\leq  d(x,{{O_{i+1}}^c})+d(x,y)< 4t.
$
Therefore,
$$K_{\Psi(t\sqrt{L})}(x,y) \chi_{(\widehat{O_{i+\!1}})^c}(y,t)=K_{\Psi(t\sqrt{L})}(x,y),  \quad {\rm for}   \ \   t>a^{(3)}.
$$

\medskip

\noindent
{\bf Case 3:    $\chi_{3}(y,t)=\chi_{{\mathcal R}(E_i^k)}(y,t)=1$}.

\smallskip

In this case, we consider two cases: $d(x, E_i^k)=0$ and $d(x, E_i^k)>0$.

\smallskip

\noindent
{\bf Subcase 3.1:}  $d(x, E_i^k)=0$.

\smallskip

  It follows that
$
d(y, {E_i^k})\leq  d(x,y)<t<2t.
$
Hence,
$$K_{\Psi(t\sqrt{L})}(x,y) \chi_{{\mathcal R}(E_i^k)}(y,t)=K_{\Psi(t\sqrt{L})}(x,y),   \quad   {\rm for} \  \  t>0.
$$
In this case, we can choose  $b^{(3)}$ and $a^{(4)}$ to be any positive number. For example,
we let $b^{(3)}=b^{(1)}$ and $a^{(4)}=a^{(2)}$.

\smallskip

\noindent
{\bf Subcase 3.2:}   $d(x, E_i^k)>0$.

\smallskip

We   choose $b^{(3)}=d(x, E_i^k)/3$ and $a^{(4)}=d(x, E_i^k)$.  If  $t<b^{(3)} $, then
$
d(y, E_i^k)\geq d(x, E_i^k)-d(x,y)> 3t-t=2t.
$
Hence,
$$K_{\Psi(t\sqrt{L})}(x,y) \chi_{{\mathcal R}(E_i^k)}(y,t)=0  \quad {\rm for}   \ \  t<b^{(3)}.
$$
If $t>a^{(4)}$, then
$
d(y, E_i^k)\leq  d(x,E_i^k)+d(x,y)< 2t.
$
This tells us that for $   t>a^{(4)},$
$$K_{\Psi(t\sqrt{L})}(x,y) \chi_{{\mathcal R}(E_i^k)}(y,t)=K_{\Psi(t\sqrt{L})}(x,y).
$$

\smallskip

\noindent
{\bf Case 4:    $\chi_{4}(y,t)=\chi_{({\mathcal R}(E_i^{k-1}))^c}(y,t)=1$}.

\smallskip

In this case, we consider two cases: $d(x, E_i^{k-1})=0$ and $d(x, E_i^{k-1})>0$.

\smallskip

\noindent
{\bf Subcase 4.1:}   $d(x, E_i^{k-1})=0$.

\smallskip

 It follows that
$
d(y, E_i^{k-1})\leq  d(x,y)<t<2t.
$
Hence,
$$K_{\Psi(t\sqrt{L})}(x,y) \chi_{({\mathcal R}(E_i^{k-1}))^c }(y,t)=0,   \quad   {\rm for} \  \  t>0.
$$
We  let $b^{(4)}$ and $a^{(5)}$ be any positive number. For example,
we choose $b^{(4)}=b^{(1)}$ and $a^{(5)}=a^{(2)}$.

\smallskip

\noindent
{\bf Subcase 4.2:}   $d(x, E_i^{k-1})>0$.

\smallskip

 We  choose $b^{(4)}=d(x, E_i^{k-1})/3$
and  $a^{(5)}=d(x, E_i^{k-1})$.  If  $t< b^{(4)}$, then
$
d(y, E_i^{k-1})\geq d(x,E_i^{k-1})-d(x,y)> 3t-t=2t.
$
This tells us
$$K_{\Psi(t\sqrt{L})}(x,y) \chi_{({\mathcal R}(E_i^{k-1}))^c }(y,t)=K_{\Psi(t\sqrt{L})}(x,y)
 \quad {\rm for}   \ \  t< b^{(4)}.
$$
If $t> d(x, E_i^{k-1})$, then
$
d(y,E_i^{k-1})\leq  d(x, E_i^{k-1})+d(x,y)< 2t.
$
Therefore,
$$K_{\Psi(t\sqrt{L})}(x,y) \chi_{ ({\mathcal R}(E_i^{k-1}))^c}(y,t)=0,  \quad {\rm for}   \ \   t>a^{(5)}.
$$

From   {\bf Cases 1, 2, 3} and  {\bf 4}  above,  we  have obtained our claim, and then
  the proof of  Lemma  \ref{le4.2} is complete.
\end{proof}

\medskip

\noindent {\it Back to the proof of (ii) of Theorem~\ref{th1.1}.}
\ We continue to  show   \eqref{e4.5}.
  \ Note that the conditions $d(x,y)<t$ and $B(y,4t)\in O_i$
imply that $x\in O_i$.  If $x\in {O_i}^c$, then
$$ \int_0^\infty\!\!\int_{{X}}
K_{\Psi(t\sqrt{L})}(x,y) \chi_{T_{i}^k}(y,t)t^2Le^{-t^2L}f(y) \,{d\mu(y)}\frac{{\rm d}t}{t}=0.
$$
Fix $x\in O_i$. We apply Lemma \ref{le4.2} to write
\begin{eqnarray}\label{e4.7}
 &&\hspace{-1cm}\int_0^\infty\int_{{X}}
K_{\Psi(t\sqrt{L})}(x,y) \chi_{T_{i}^k}(y,t)t^2Le^{-t^2L}f(y) \,{d\mu(y)}\frac{{\rm d}t}{t}\\
&=&\left\{\int_0^{b_0}+\sum_{l=1}^{N-1}\int_{a_l}^{b_l}\right\}\int_{{X}}
K_{\Psi(t\sqrt{L})}(x,y)\chi_{T_{i}^k}(y,t)t^2Le^{-t^2L}f(y) \,{d\mu(y)}\frac{{\rm d}t}{t} \nonumber\\
&+& \left\{\sum_{l=0}^{N-1}\int_{b_l}^{a_{l+1}}\right\} \int_{{X}}
K_{\Psi(t\sqrt{L})}(x,y)\chi_{T_{i}^k}(y,t)t^2Le^{-t^2L}f(y) \,{d\mu(y)}\frac{{\rm d}t}{t} \nonumber\\
&=&I_1(x)+I_2(x).\nonumber
\end{eqnarray}
To estimate $I_1(x)$, we note that if  $0\leq a<b\leq b_1$ or $a_l\leq a<b\leq b_l$, then  one has either
$$\int_{a}^{b}\int_{{X}}
K_{\Psi(t\sqrt{L})}(x,y) \chi_{T_{i}^k}(y,t)t^2Le^{-t^2L}f(y) \,{d\mu(y)}\frac{{\rm d}t}{t}=0,
$$
or by \eqref{e4.2},
\begin{eqnarray*}
 \int_{a}^{b}\int_{{X}}
K_{\Psi(t\sqrt{L})}(x,y) \chi_{T_{i}^k}(y,t)t^2Le^{-t^2L}f(y) \,{d\mu(y)}\frac{{\rm d}t}{t}
 &=&\int_{a}^{b}
\Psi(t\sqrt{L})t^2Le^{-t^2L}f(x) \,\frac{{\rm d}t}{t}\\
&=&\eta(a\sqrt{L})f(x)-\eta(b\sqrt{L})f(x).
\end{eqnarray*}
Observe that for each $a\leq t\leq b$, if $d(x,y)<t$, then $\chi_{T_{i}^k}(y,t)=1$.
This tells us that $(y,t)\in (\widehat{O_{i+1}})^c$, hence $B(y,4t)\cap \big(O_{i+1}\big)^c\neq \emptyset$.
Assume that ${\bar x}\in B(y,4t)\cap (O_{i+1})^c$. From this, we have that $d(x,\bar{x})
\leq d(x,y)+d(y,\bar{x})< 5t$ and ${\mathscr M}_{L}f(\bar{x})\leq 2^{i+1}$.
It implies that $|\eta(t\sqrt{L})f(x)|\leq {\mathscr M}_{L}f(\bar{x})\le C2^{i+1}$
for every $a\leq t\leq b$.
Therefore,  $|\eta(a\sqrt{L})f(x)|\le C2^{i+1}$ and
$|\eta(b\sqrt{L})f(x)|\le C2^{i+1}$, and so $|I_1(x)|\leq C2^{i+1}$.

Consider   $I_2(x)$.   If $\chi_{T_{i}^k}(y,t)=1$,
then $(y,t)\in (\widehat{O_{i+1}})^c$. Thus $B(y,4t)\cap (O_{i+1})^c\neq \emptyset$.
Assume that $\bar{x}\in B(y,4t)\cap (O_{i+1})^c$. We have that $|t^2Le^{-t^2L}f(y)|\leq
{\mathscr M}_{L}f(\bar{x})\le 2^{i+1}$. This, together with $a_{l+1}\leq 3^4b_l \  (l=0,1,\cdots, N-1),$
 implies that
\begin{eqnarray}\label{e4.8}
 \Big|\int_{b_l}^{a_{l+1}}\int_{{X}}
K_{\Psi(t\sqrt{L})}(x,y) \chi_{T_{i}^k}(y,t)t^2Le^{-t^2L}f(y) \,{d\mu(y)}\frac{{\rm d}t}{t}\Big|
&\leq& 2^{i+1}\Big|\int_{b_l}^{81b_l}\int_{{X}}
|K_{\Psi(t\sqrt{L})}(x,y)|  \,{d\mu(y)}\frac{{\rm d}t}{t}\Big|\nonumber\\
&\leq& C 2^{i+1} \int_{b_l}^{81b_l}\frac{1}{t}\, dt\leq  C 2^{i+1},
\end{eqnarray}
which yields that $|I_2(x)|\leq C 2^{i+1}$.

Combining  estimates of $I_1(x)$ and $I_2(x)$, we  have obtained (\ref{e4.5}).
This proves (ii) of Theorem~\ref{th1.1}, and the proof of Theorem~\ref{th1.1} is complete.
  \hfill{}$\Box$

\smallskip

  \medskip

\noindent
{\bf Remarks.} \ i) In \cite{DKKP}, the authors established the equivalence of the maximal
and atomic Hardy spaces associated to an operator $L$ on the space of homogeneous type $X$,
\iffalse
i.e.,
$$
 H^p_{L, { max}}({X}) \simeq  H^p_{L, {  rad}}({X})  \simeq  H^p_{L, {\rm at}, q, M}({X}),
$$\fi
 under the following four assumptions that  $L$ satisfies
 ${\bf (H1)}$-${\bf (H2)}$,
and

\smallskip

\noindent
{\bf (H3)} \ the kernel $p_t(x,y)$ of the semigroup $e^{-tL}$ satisfies  the
H\"older continuity: There exists a constant $\alpha>0$
such that
% on space variable $y$
$$
 |p_t(x,y)-p_t(x,y')|\leq C  \left({d(y,y')\over \sqrt{t}}\right)^{\alpha}
 %\frac{1}{V(x,\sqrt{t})} \exp\left(-{  {{d(x,y)}^2}\over  ct}\right)
$$
 for $x, y, y'\in X$ and $t>0$, whenever $d(y, y')\leq \sqrt{t}$; and

 \noindent
{\bf (H4)} \ Markov property:
$$
\int_X p_t(x,y)d\mu(y)=1
$$
for $x\in X$  and $t>0.$

 \medskip

 ii) \ Let $L$ be  an operator satisfying ${\bf (H1)}$-${\bf (H2)}$. For $f\in L^2(X)$,
we define  an area function  $S_{L}f$  associated to the heat semigroup
generated by $L$,
\begin{eqnarray}\label{e5.1}
S_{L}f(X):=\left(\int_0^{\infty}\!\!\int_{d(x,y)<t}
\big|t^2Le^{-t^2L} f(y)\big|^2 {d\mu(y) dt\over   t V(y,t)}\right)^{1/2},
\quad x\in X.
\end{eqnarray}
 Given $0<p\leq 1$. The Hardy space $H^p_{L,  S}(\RR^n) $ is defined as the completion of
 $\{ f\in L^2(X):\, \|S_{L} f\|_{L^p(X)}<\infty \}
 $
 with  norm
$$
 \|f\|_{H^p_{L, S}(X)}:=\|S_{L} f\|_{L^p(X)}.
$$
From  \cite{DL}, \cite{HLMMY} and   Theorem~\ref{th1.1} above,
 $$
  H^p_{L, {\rm at}, q, M}({X}) \simeq  H^p_{L, S}(X)\simeq
  H^p_{L, {  max}}({X}) \simeq  H^p_{L, {  rad}}({X})
$$
 for every $0<p\leq 1$, and for all $q>p$ with $1\leq q\leq \infty$ and   all integers $M> {n\over 2}({1\over p}-1)$.

\iffalse
ii) Suppose that an operator satisfies ${\bf(H 1)}$ and the following assumption ${\bf(H 3)}$:

${\bf(H 3)}$ The operator $L$ generates an analytic semigroup
$\{e^{-tL}\}_{t>0}$ which satisfies the Davies-Gaffney condition.
That is, there exist constants $C$, $c>0$ such that for any open subsets
$U_1,\,U_2\subset X$,

\begin{equation}\label{e2.5}
|\langle e^{-tL}f_1, f_2\rangle|
\leq C\exp\Big(-{{\rm dist}(U_1,U_2)^2\over c\,t}\Big)
\|f_1\|_{L^2(X)}\|f_2\|_{L^2(X)},\quad\forall\,t>0,
\end{equation}

\noindent for every $f_i\in L^2(X)$ with
$\mbox{supp}\,f_i\subset U_i$, $i=1,2$,
where ${\rm dist}(U_1,U_2):=\inf_{x\in U_1, y\in U_2} d(x,y)$.
\fi

\bigskip

\noindent
{\bf Acknowledgments.}
 L. Song is supported by  NNSF of China (Grant Nos. 11471338 and ~11521101) and  Natural Science Funds for Distinguished Young Scholar of Guangdong Province.   L.  Yan is  supported by
 NNSF of China (Grant Nos.  11371378 and  ~11521101).

\end{document}